\documentclass[preprint,3p,twocolumn]{elsarticle}
\usepackage{amsmath, amsthm, amsfonts, amssymb, graphicx, cite, epsfig,epstopdf,color,soul}

\newtheorem{theorem}{Theorem}
\newtheorem{proposition}{Proposition}
\newtheorem{lemma}{Lemma}
\newtheorem{corollary}{Corollary}
\theoremstyle{definition}

\newtheorem{assumption}{Assumption}
\theoremstyle{remark}
\newtheorem{remark}{Remark}

\begin{document}
\begin{frontmatter}
\title{\textbf{Distributed Resource Allocation on Dynamic Networks in Quadratic Time}}
\author[rvt]{Thinh Thanh Doan\corref{cor1}}
\ead{ttdoan2@illinois.edu}
\author[ise]{Alex Olshevsky\corref{ack}}
\ead{aolshev2@illinois.edu}

\address[rvt]{Deparment of Electrical and Computer Engineering, University of Illinois,  Urbana, IL, USA}
\address[ise]{Department of Industrial and Enterprise Systems Engineering, University of Illinois,  Urbana, IL, USA}
\cortext[cor1]{Corresponding author}
\cortext[ack]{This research was supported by NSF grant CMMI-1463262.}
\begin{abstract} We consider the problem of allocating a fixed amount of resource among nodes in a network when each node suffers a cost which is a convex function of the amount of resource allocated to it. We propose a new deterministic and distributed protocol for this problem. Our main result is that the associated convergence time for the global objective scales quadratically in the number of nodes on any sequence of time-varying undirected graphs satisfying a long-term connectivity condition.
\end{abstract}
\end{frontmatter}
\section{Introduction}

We consider the problem of optimally allocating a fixed amount of resource among $n$ agents. Each agent suffers a convex cost as a function of the amount of resource allocated to it and the goal is to distribute the resource among the agents to minimize the total cost incurred. Sometimes the problem is described in terms of utilities, with each agent having a concave utility function and the goal being to maximize the total utility.

Our goal is to develop distributed protocols for this problem, meaning that nodes are only allowed to interact with neighbors in some graph or some time-varying sequence of graphs. Our motivation comes from  potential applications in sensor networks, which regularly face the problems of optimally allocating communication bandwidth and computing power \citep{wireless}. Furthermore, resource allocation is a simplification of the important ``economic dispatch'' problem wherein geographically distributed producers of electricity must coordinate to meet a fixed demand \citep{CavaroCarliZampieri14, CortesA, JorgeA}.

The problem  has an old history dating back to the classic work of Arrow and Hurwicz \citep{ArrowHurwicz60}. The first algorithm which could be implemented in a distributed way was the ``center-free'' protocol of \citep{HoServiSuri80}. In the protocol of \citep{HoServiSuri80}, each node increases the amount of resource allocated to it proportionally to the difference in  gradients between its neighbors and itself. It was shown in \citep{HoServiSuri80} that, under appropriate technical conditions, this protocol will drive the amount of resource allocated to each node to the optimal value. The term ``center-free'' was originally meant to refer to the absence of any central coordinating authority, though in this paper we will use it to mean any update wherein nodes update the amount of resource by looking at gradient differences with neighbors. The work of \citep{HoServiSuri80} has spawned a number of modern follow-ups, including  \citep{KuroseSimha89,XiaoBoyd06,HoServiSuri80,LakshmananDeFarias08, Nec-others, Necoara13} as well as the current paper.

The paper \citep{KuroseSimha89} considered the resource allocation problem in the context of optimal distribution of a database among the nodes of the network; some modifications of the algorithm of \citep{HoServiSuri80} which used not only gradient differences but also the second derivatives of the cost functions were proposed. More recently,  \citep{XiaoBoyd06} studied the case when the cost functions are strongly convex and noted that the problem of optimal weight selection for center-free methods can be cast as a semidefinite program. The work of \citep{LakshmananDeFarias08} analyzed a natural class of center-free methods on time-varying networks and provided a convergence analysis. The recent paper \citep{Nec-others, Necoara13} studied the convergence rates of distributed protocols which repeatedly choose a random pair of neighboring nodes and perform a center-free update on that pair. Finally, the work of \citep{Angelia14}  used accelerated gradient methods to design distributed protocols for a more general  problem.

Our focus in this paper is on designing protocols with good convergence speed. Specifically, we are interested at how the gap to the optimal objective value scales in the worst-case with iteration $k$ and the number of nodes $n$ in the system.

The best previously known results were provided in the antecedent papers \citep{LakshmananDeFarias08, Necoara13}. Both papers considered the class of costs which have Lipschitz-continuous derivatives. The paper \citep{Necoara13} considers schemes which randomly pick pairs of neighbors to perform a center-free update; if the pairs are chosen uniformly at random the convergence time implied by the results of \citep{Necoara13} is $O(L n^4/k)$ in expectation\footnote{The convergence rate in \citep{Necoara13} is given in terms of the eigenvalues of a certain matrix; the quartic bound above follows by putting \citep{Necoara13} together with the well-known fact that the smallest eigenvalue of the Laplacian of a connected, undirected graph on $n$ is $\Omega(1/n^2)$.} on fixed graphs; here $L$ is the largest of the Lipschitz constants of the derivatives of the cost functions. However, we note here that it is possible to shave off a factor of $n$ off this bound by adjusting the probabilities in a graph-dependent way. The paper \citep{LakshmananDeFarias08} does not give an explicit convergence rate for the objective, but  gives a worst-case $O(LBn^3/k)$ rate for the decay of the average of {\em squared} gradient differences in the graph;  here $B$ is a constant which measures how long it takes for the time-varying graph sequence to reach connectivity. Improved rates were obtained in \citep{Necoara15} and in \citep{Beck14} for a more general problem, but  under the assumption that the graph is a fixed complete graph.

In this paper, we show a convergence rate of $O \left(LBn^2/k \right)$ for the objective under the same assumptions of Lipschitz-continuous derivatives in the more general setting of time-varying graphs. Additionally, when the costs are strongly convex, we demonstrate a geometric rate of $ O \left( \left(1 - \mu/(4Ln^2) \right)^{ k/B} \right)$ where $\mu$ is the parameter of strong convexity. For both of these rates, the number of iterations until the objective is within $\epsilon$ of its optimal value scales quadratically with the number of nodes $n$. This is an improvement over the results described above, though we note that our protocol involves every node contacting its neighbors and performing an update at every step (which involves $O(|E(t)|)$ messages exchanged, where $E(t)$ is the set of edges at time $t$, and $O(n)$ updates); whereas \citep{Necoara13} relied on only a pair of randomly chosen nodes updating at each step.

The remainder of this paper is organized as follows. We give a formal statement of the problem in Section \ref{sec:ProblemFormulation}. Our protocol is described in
Section \ref{sec:GradientBalacingAlgorithm}. The convergence analysis of the protocol is
in Section \ref{sec:ConvergenceAnalysis}. Finally,  Section \ref{sec:Simulation} describes the results of some simulations and we conclude
in Section \ref{sec:Conclusion}.

\section{Problem Formulation}\label{sec:ProblemFormulation}
In this paper, we study distributed protocols for the following minimization problem,
\begin{align}
      \min &~~ \sum_{i=1}^n f_i(x_i)\label{prob:MainProblem}\\
      \text{s.t.} &~~ \sum_{i=1}^n x_i =K.\nonumber
\end{align} We assume that there are $n$ agents or nodes which we will index as $1, \ldots, n$,  that $f_i:\mathbb{R}\rightarrow\mathbb{R}$ is a convex function known only to node $i$ and $x_i\in\mathbb{R}$ is a variable stored by node $i$,  and finally that $K$ is some nonnegative number.

As remarked, this models a resource allocation problem among $n$ agents: given a finite amount $K$ of a certain resource, we must allocate it among agents $1, \ldots, n$ in an optimal way. 

For simplicity, we introduce notation for the total objective function $F({\bf x}) = \sum_{i=1}^n f_i(x_i),$ and the feasible set $\mathcal{S} = \{\mathbf{x}\in\mathbb{R}^n:\sum_{i=1}^{n}x_i=K\}$.

We assume that a sequence of time-varying undirected graphs models the communication between the nodes.  Specifically, we assume we are given a sequence of  undirected graphs $\mathcal{G}(k) = (\mathcal{V},\mathcal{E}(k))$ with $\mathcal{V} = \{1,\ldots,n\}$; nodes $i$ and $j$ can send exchange messages at time $k$ if and only if $(i,j) \in \mathcal{E}(k)$. We denote by $\mathcal{N}_i(k)$ the set of neighbors of node $i$ at time $k$.

We make the following fairly standard assumption which ensures that the graph sequence $\mathcal{G}(k)$ satisfies a long-term connectivity property.

\begin{assumption}\label{assum:BConnectivity}
There exists an integer $B\geq1$ such that the undirected graph
\begin{equation}
(\mathcal{V},\mathcal{E}(\ell B)\cup \mathcal{E}(\ell B+1)\cup\ldots\cup \mathcal{E}((\ell +1)B-1))
\end{equation}
is connected for all nonnegative integers $\ell $.
\end{assumption}

We will also be assuming that each local objective function $f_i(\cdot)$ is differentiable with  Lipschitz continuous derivative.

\begin{assumption}\label{assum:LipschitzCondition}
For each  $i=1, \ldots, n$, the function $f_i(\cdot)$ is differentiable everywhere and there exists a constant $L_i$ such that
\begin{equation}
|f_i'(y_i)-f_i'(x_i)| \leq L_i|y_i-x_i|,\; ~~\forall x_i,y_i\in\mathbb{R}.\nonumber
\end{equation}
\end{assumption}

\smallskip

Moreover, we will be assuming that there exists at least one optimal solution.

\begin{assumption}\label{assum:ExistenceSolution}
There exists a vector $\mathbf{x}^*=(x_1^*,x_2^*,\ldots,x_n^*)$ with ${\bf x}^* \in \mathcal{S}$ which achieves the minimum in problem \eqref{prob:MainProblem}.
\end{assumption}

We will use $\mathcal{X}^*$ to denote the set of optimal solutions to problem \eqref{prob:MainProblem}; the previous
assumption ensures that $\mathcal{X}^*$ is not empty.

Finally, we will be assuming that our algorithm starts from a feasible point.

\begin{assumption} \label{assumpt:feasible} ${\bf x}(0) \in \mathcal{S}$.
\end{assumption}

{\em For the remainder of this paper, we will be assuming that Assumptions 1,2,3, 4 hold without mention.}

We conclude this section with a characterization of the points in the optimal set $\mathcal{X}^*$; the proof is immediate.

\begin{proposition}\label{prop:OptimalCondition} We have that ${\bf x} \in \mathcal{X}^*$ if and only if ${\bf x} \in \mathcal{S}$ and $f_i'(x_i)=f_j'(x_j)$ for all $i,j \in \{1, \ldots, n\}$.
\end{proposition}

\section{Main Algorithm}\label{sec:GradientBalacingAlgorithm}
In this section, we will introduce a distributed protocol, which we call the {\sl gradient balancing protocol}, to solve problem \eqref{prob:MainProblem}. Before giving a statement of the algorithm, we provide some brief motivation for its form.

Previous protocols for problem \eqref{prob:MainProblem}
 tended to be ``center-free'' updates \citep{HoServiSuri80, XiaoBoyd06, LakshmananDeFarias08, Necoara13} where node $i$ updated as \begin{small}
\begin{equation}\label{eq:WeightedGradientDescent}
x_i(k+1)\!=\!x_i(k)-\!\!\!\!\sum_{i\in\mathcal{N}_i(k)}\!\!\!w_{ij} \left(f_i'(x_i(k))\!-\!f_j'(x_j(k)) \right),
\end{equation} \end{small}
where $w_{ij}$ is a collection of nonnegative weights.  The protocol of
\citep{Necoara13}
had a different form but proceeded in the same spirit; in that protocol, edges were repeatedly chosen according to some probability distribution and a form of the above update was performed by the incident nodes.

The protocol we propose in this paper speeds up this update by employing some local ``pruning'' wherein each node tries to perform a version of Eq. (\ref{eq:WeightedGradientDescent}), but only with the two nodes whose derivative is largest and smallest in its neighborhood. Thus nodes essentially ignore neighbors whose derivatives are  close to their own.  Intuitively, by focusing on nodes whose derivatives are far apart we increase the speed at which information propagates through the network. The idea has been previously used in  \citep{NedicOlshevskyOzdaglarTsitsiklis09} and is inspired by an algorithm from Chapter 7.4 of \citep{TBbook}.

We now describe the steps node $i$ executes at step $k$ to update its value from $x_i(k)$ to $x_i(k+1)$. We assume that all nodes execute these steps synchronously, and furthermore that all four steps of the protocol given below can be executed before the graph changes from $\mathcal{G}(k)$ to $\mathcal{G}(k+1)$. Speaking informally, the protocol consists of each node repeatedly trying to ``match'' itself to the node in its neighborhood whose derivative is smallest and smaller than its own in order to perform a center-free update.

\vspace{5mm} \hspace{-6mm}
\textbf{The Gradient Balancing Protocol}

\begin{enumerate}[\leftmargin=13mm]
\item[1.] Node $i$ broadcasts its derivative $f'_i(x_i(k))$ and Lipschitz constant $L_i$ to its neighbors.
\smallskip
\item[2.]
Going through the messages it has received from neighbors as a result of step 1, node $i$ finds the neighbor with the smallest derivative that is strictly less than its own.
Let $p$ be a neighbor with this derivative; ties can be broken arbitrarily. Formally, \begin{small}  $$p \in {\arg \min_{j}} \{  f_j'(x_j(k)) ~~|~~  j \in N_i(k), ~~~ f_j'(x_j(k)) < f_i'(x_i(k)) \}.$$  \end{small} Node $i$ then sends a message containing the number $$\Delta_{ip}(k) = \frac{1}{2} \frac{f'_i(x_i(k))-f'_p(x_p(k))}{L_i+L_p}$$ to node $p$.\\ \\
If no neighbor of $i$ has a derivative strictly less
than $f'_i(x_i(k))$, node $i$ does nothing during this step.
\smallskip
\item[3.] Node $i$ goes through any $\Delta_{ji}(k)$ it has just received from its neighbors $j \in N_i(k)$ as a result of step 2, and finds the largest among them; ties can be broken arbitrarily. Let us suppose this is $\Delta_{qi}(k)$. \\ \\
Node $i$ then sets $$y_i(k) = x_i(k) + \Delta_{qi}(k).$$  Furthermore, node $i$ sends an ``accept'' message to node $q$ and a ``reject'' message to any other neighbor $j$ that sent it a $\Delta_{ji}(k)$ in step 2. \\ \\
If node $i$ did not receive any $\Delta_{ji}(k)$ as a result of step 2, it sets $y_i(k) = x_i(k)$.
\medskip
\item[4.] If node $i$ did not send out $\Delta_{ip}(k)$ during step 2, or if it received a ``reject'' from the node $p$ to whom
it sent $\Delta_{ip}(k)$,
it sets $x_i(k+1) = y_i(k)$. \\
If node $i$ has received an ``accept'' from node $p$, it sets  $$x_i(k+1) = y_i(k) - \Delta_{ip}(k).$$
\end{enumerate}
\bigskip

\medskip

Informally, we will refer to the numbers $\Delta_{ij}(k)$ as ``offers.'' We may summarize the gradient balancing protocol as follows. Each node $i$ makes an offer to the node with the smallest derivative (below its own) in its neighborhood, and the size of the offer is proportional to the difference of the derivatives normalized by the sum of the respective Lipschitz constants. Each node then accepts the largest offer it has received and rejects the rest. Note that each node ``accepts'' at most one offer and ``makes'' at most one offer. The final result is something like Eq. (\ref{eq:WeightedGradientDescent}), except that the graph has been pruned to be of degree at most two and contain only edges between nodes whose derivatives are ``far apart.''

We remark that an immediate consequence of Assumption \ref{assumpt:feasible} is that  ${\bf x}(k) \in \mathcal{S}$ for all $k \geq 0$, since every accepted offer involves an increase at the receiving node and a decrease at the offering node of the same magnitude.

For concreteness, we provide an example of our protocol; see Fig.
\ref{fig:AlgorithmSimulation}. The top part of the figure shows $x_i(k)$ and $f_i'(x_i(k))$ for each node in parenthesis, respectively. We assume that $L_i=1/2$ for all $i=1, \ldots, n$. The bottom part of the figure shows the new values $x_i(k+1)$.  As we can see that node $B$ and node $C$ send offers to node $D$ but
node $D$ only accepts node $B$'s offer. Node $D$ also sends an offer to node $E$ and node $E$ accepts since it is the only offer it receives. Nodes $A$ and $C$ do not end up participating in any accepted offers and consequently for those nodes $x_i(k+1)=x_i(k)$.

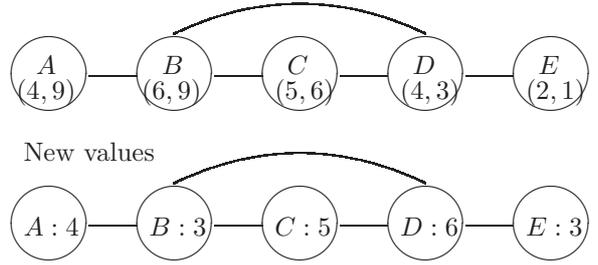
\begin{figure}[h]
\setlength{\unitlength}{0.13in} 
\centering 
\begin{picture}(50,15)(0,0) 
\put(3,6){\circle{3}}\put(2,5.5){$A:4$}\put(4.6,5.9){\line(1,0){1.9}}
\put(8,6){\circle{3}}\put(7,5.5){$B:3$} \put(9.6,5.9){\line(1,0){1.9}}
\qbezier(18,7.6)(13,10)(8,7.6)\thinlines

\put(13,6){\circle{3}}\put(12,5.5){$C:5$}\put(14.6,5.9){\line(1,0){1.9}}
\put(18,6){\circle{3}}\put(17,5.5){$D:6$}\put(19.6,5.9){\line(1,0){1.9}}
\put(23,6){\circle{3}}\put(22,5.5){$E:3$}

\put(3,12){\circle{3}}\put(2.5,12){$A$}\put(1.7,11){$(4,9)$} \put(4.6,11.9){\line(1,0){1.9}}
\put(8,12){\circle{3}}\put(7.5,12){$B$}\put(6.7,11){$(6,9)$} \put(9.6,11.9){\line(1,0){1.9}}
\qbezier(18,13.6)(13,16)(8,13.6)\thinlines

\put(2,8.5){\text{New values}}

\put(13,12){\circle{3}}\put(12.5,12){$C$}\put(12,11){$(5,6)$} \put(14.6,11.9){\line(1,0){1.9}}
\put(18,12){\circle{3}}\put(17.5,12){$D$}\put(17,11){$(4,3)$} \put(19.6,11.9){\line(1,0){1.9}}
\put(23,12){\circle{3}}\put(22.5,12){$E$}\put(22,11){$(2,1)$}
\end{picture}
\vspace{-50pt} \caption{A  step of the gradient balancing
protocol.}\label{fig:AlgorithmSimulation}
\end{figure}

\section{Convergence Analysis}\label{sec:ConvergenceAnalysis} We now turn to the convergence analysis of the gradient balancing protocol. We will prove upper bounds on $F({\bf x}(k)) - {F}({\bf x}^*)$ which imply that the time until this quantity shrinks below $\epsilon$ is quadratic in the number of nodes $n$.

We start with the observation that the gradient balancing protocol may be rewritten in a particularly convenient way. Let us define $\overline{\mathcal{E}}(k) $ to be the set of pairs $(i,j)$ such that either $i$ accepts an offer from $j$ at time $k$ or vice versa. We can then write  \begin{align}\label{eq:iUpdate}
x_i(k+1)= x_{i}(k) -\!\!\!\!\!\!\!\sum_{j ~|~ (i,j)\in\overline{\mathcal{E}}(k)}\!\!\!\!\!\!\!~~\frac{f_i'(x_i(k)) - f_j'(x_j(k))}{2(L_i+L_j)},
\end{align}

We now begin with a series of lemmas which lead the way to our
main convergence result. Our first lemma shows the monotonicity of the largest and smallest derivatives
in the network.

\begin{lemma} The function $\min_{i=1, \ldots, n} f_i'(x_i(k))$ is nondecreasing in $k$ and the
function $ \max_{i=1, \ldots, n} f_i'(x_i(k))$ is nonincreasing in $k$.  \label{mono}
\end{lemma}
\begin{proof} Consider node $j$.  We will show that there is always some node $q$ such that
$f_j'(x_j(k+1)) \geq f_q'(x_q(k))$. This will prove the monotonicity of the smallest derivative; the monotonicity of the
largest derivative is proved analogously.

If $j$ does not make any offers during step 2 of
the gradient balancing protocols, or if it makes an offer which is rejected, then we must have $x_j(k+1) \geq x_j(k)$; since $f_j(\cdot)$ is
convex this implies that $f_j'(x_j(k+1)) \geq f_j'(x_j(k))$. Thus we may take $q=j$ in this case.

On the other hand, suppose $j$ makes an offer during step 2 which is accepted, say by node $m$. From
Eq. (\ref{eq:iUpdate}),
\[ x_j(k+1)  \geq  x_j(k) - \frac{f_j'(x_j(k)) - f_m'(x_m(k))}{2 (L_m + L_j) }, \] and since {$f_j(\cdot)$ is convex and differentiable with $L_j$-Lipschitz continuous derivative},  \begin{small}
\begin{align*} f_j'(x_j(k+1) )& \geq  f_j' \left(  x_j(k) - \frac{f_j'(x_j(k)) - f_m'(x_m(k))}{2 (L_m + L_j) } \right)  \\
& \geq  f_j'(x_j(k)) -  \frac{L_j \left( f_j'(x_j(k)) - f_m'(x_m(k)) \right)}{2(L_m + L_j)} \\
& >  f_m'(x_m(k)),
\end{align*} \end{small} so that we may take $q=m$ in this  case.
\end{proof}

We will often need to be making statements about the $d$ largest derivatives at time $k$. To avoid overburdening the reader
with notation, we will begin many of our lemmas with a variation on the  words ``let us relabel the vertices so that the sequence
$f_1'(x_1(k)), f_2'(x_2(k), \ldots, f_n'(x_n(k))$ is nonincreasing.'' Under this assumption, the $d$ largest derivatives may be taken to be
$f_1'(x_1(k)), \ldots, f_d'(x_d(k))$.

Furthermore, assuming the nodes have been relabeled as above, we will say that {\em edge $(i,j)$ crosses the cut $d$} if one of $i,j$ belongs to
$\{1,\ldots, d\}$ while the other belongs to $\{d+1,\ldots, n\}$.

An example of the use of these definitions is in the following corollary, which is an immediate consequence of  Lemma \ref{mono}.

\begin{corollary} Let us relabel the nodes so that the sequence $f_1'(x_1(k)), f_2'(x_2(k)), \ldots, f_n'(x_n(k))$ is
nonincreasing. Suppose that during times $t = k, k+1, \ldots, k+T$ we have that $\overline{\mathcal{E}}(t)$ did not include
any edges crossing the cut $d$. Then for $i=1, \ldots, d$, we have that \[ f_i'(x_i(t+T+1))  \geq f_d'(x_d(t)), \] while for
$i=d+1, \ldots, n$, \[ f_i'(x_i(t+T+1))  \leq f_{d+1}'(x_{d+1}(t)). \] \label{order}
\end{corollary} \vspace{-0.5cm}
Our next lemma essentially says that cuts in the graph which separate larger derivatives from smaller derivatives must have
edges in $\overline{\mathcal{E}}(k)$ which cross them eventually. The proof follows  from Assumption 1 on $B$-connectivity and is  the same as the proof of Lemma 3
in \citep{NedicOlshevskyOzdaglarTsitsiklis09}, so we omit it.

\begin{lemma}\label{prop:BConnectivityRelaxation} Let $\ell  \geq 0$ and let us relabel the nodes so that the sequence
$f_1'(x_1(\ell B)), f_2'(x_2(\ell B)), \ldots, f_n'(x_n(\ell B))$ is nonincreasing. Then
for every $d\in\{1,\ldots,n-1\}$, either $f_d'(x_d(\ell B))=f_{d+1}'(x_{d+1}(\ell B))$, or there
exist some time $k\in\{\ell B,\ldots,(\ell +1)B-1\}$ when an edge in $\overline{\mathcal{E}}(k)$ crosses the cut $d$.
\end{lemma}

We now proceed to our first substantial lemma, which shows that the gradient balancing protocol is a descent protocol (meaning that
$F({\bf x}(k))$ is nonincreasing).

\begin{lemma}\label{lem:DecreasingFuncValue} \begin{small}
\begin{align}\label{eq:FunctionDecreasing}
F(\mathbf{x}(k+1))\leq F(\mathbf{x}(k)) -\!\!\!\!\!\!\sum_{(i,j)\in
\overline{\mathcal{E}}(k)}\!\!\!\!\!\!\!\frac{(f_{i}'(x_i(k))-f_{j}'(x_j(k)))^2}{4(L_{i}+L_{j})}\cdot
\end{align} \end{small}
\end{lemma}
\begin{proof} Assumption 2 immediately implies that for all $x_i,y_i\in\mathbb{R}$
\begin{equation}
f_i(y_i)\leq f_i(x_i) + f_i'(x_i)(y_i-x_i) + \frac{L_i}{2}(y_i-x_i)^2.\nonumber
\end{equation}
Summing up both sides over $i=1, \ldots, n$, we obtain
\begin{equation}
F(\mathbf{y})\leq F(\mathbf{x}) + \sum_{i=1}^n f_i'(x_i)(y_i-x_i) + \sum_{i=1}^n \frac{L_i}{2}(y_i-x_i)^2.\nonumber
\end{equation}
Replacing $\mathbf{x}$ by $\mathbf{x}(k)$, $\mathbf{y}$ by $\mathbf{x}(k+1)$, we obtain
\begin{align}\label{eq:Inequality1}
F(\mathbf{x}(k+1))&\leq\; F(\mathbf{x}(k))\nonumber\\
&+\sum_{i=1}^n f_i'(x_i(k))(x_i(k+1)-x_i(k)) \nonumber\\
&+\sum_{i=1}^n \frac{L_i}{2}(x_i(k+1)-x_i(k))^2.
\end{align}
On the other hand, one consequence of Eq. \eqref{eq:iUpdate} is that \begin{small}
\begin{align}\label{eq:Inequality2}
&\sum_{i=1}^n f_i'(x_i(k))(x_i(k+1)-x_i(k))\nonumber\\
&= -\sum_{i=1}^n ~\sum_{j ~|~ (i,j)\in\overline{\mathcal{E}}(k)} ~ \frac{f_i'(x_i(k))}{2(L_i+L_j)}(f_i'(x_i(k)) - f_j'(x_j(k)))\nonumber\\
&= -\sum_{(i,j)\in\overline{\mathcal{E}}(k)}\frac{(f_i'(x_i(k)) - f_j'(x_j(k)))^2}{2(L_i+L_j)}.
\end{align} \end{small}
Furthermore, another consequence of Eq. \eqref{eq:iUpdate} is that
\begin{align}\label{eq:Inequality3}
&\sum_{i=1}^n \frac{L_i}{2}(x_i(k+1)-x_i(k))^2\nonumber\\
&=\sum_{i=1}^n \frac{L_i}{2}\Bigg(\sum_{j ~|~ (i,j)\in\overline{\mathcal{E}}(k)}\frac{f_i'(x_i(k)) - f_j'(x_j(k))}{2(L_i+L_j)}\Bigg)^2\nonumber\\
&\leq\sum_{i=1}^n ~ \sum_{j ~|~ (i,j)\in\overline{\mathcal{E}}(k)}\frac{2L_i(f_i'(x_i(k)) - f_j'(x_j(k)))^2}{8(L_i+L_j)^2}\nonumber\\
&=\sum_{(i,j)\in\overline{\mathcal{E}}(k)}\frac{(f_i'(x_i(k)) - f_j'(x_j(k)))^2}{4(L_i+L_j)},
\end{align} where the inequality follows because  node $i$ is incident to at most two edges in $\overline{\mathcal{E}}(k)$, which allows us
to use the inequality $(a+b)^2 \leq 2 (a^2 + b^2)$. Finally,
we substitute Eq. \eqref{eq:Inequality2} and Eq. \eqref{eq:Inequality3}
into Eq. \eqref{eq:Inequality1} to obtain the statement of the lemma.
\end{proof}

Glancing at Eq. \eqref{eq:FunctionDecreasing} which we have just derived, we might guess that the second
term on the right is ultimately going to determine how fast the gradient balancing protocol will converge.
Our next two lemmas provide useful lower bounds for this
quantity over the time interval $k=\ell B, \ldots, (\ell+1)B-1$.

\begin{lemma}\label{lem:BoundDiffGradOnBInterval} Let us relabel the nodes so that
the sequence $f_1'(x_1(\ell B)), f_2'(x_2(\ell B)), \ldots, f_n'(x_n(\ell B))$ is in nonincreasing order. Then, \begin{small}
\begin{align}
&\sum_{k=\ell B}^{(\ell +1)B-1}\sum_{(i,j)\in
\overline{\mathcal{E}}(k)}(f_i'(x_i(k))-f_j'(x_j(k)))^2\nonumber\\
&\geq\sum_{d=1}^{n-1}(f_d'(x_d(\ell B))-f_{d+1}'(x_{d+1}(\ell B)))^2\nonumber.
\end{align} \end{small}
\end{lemma}
\begin{proof} We begin the proof by introducing some notation. For all $k \in \{\ell B, \ell B+1, \ldots, (\ell +1)B-1\}$, we use $D(k)$ to denote the set of $d \in\{1,\ldots,n-1\}$ such that time $k$ is the first time in $\{\ell B, \ell B+1, \ldots, ({l}+1)B-1\}$ with an edge $(i,j) \in \overline{\mathcal{E}}(k)$ crossing the cut $d$. Note that $D(k)$ may be empty. Furthermore, given the edge $(i,j)  \in\overline{\mathcal{E}}(k)$ we will use $F_{ij}(k)$ to denote all the cuts $d \in D(k)$ crossed by $(i,j)$ at time $k$. Likewise, it may be the case that $F_{ij}(k)$ is empty.

We begin with the following observation. Suppose $F_{ij}(k)=\{d_1,\ldots,d_q\}$ where $d_1 < d_2 < \cdots < d_q.$ Then \vspace{-0.3cm}
\begin{small} \begin{align} \label{ffineq}
& (f_i'(x_i(k))-f_j'(x_j(k)))^2   \cr
& \geq \sum_{d\in F_{ij}(k)} \left( f_d'(x_d(\ell B))-f_{d+1}'(x_{d+1}(\ell B)) \right)^2.
\end{align} \end{small}We now justify Eq. (\ref{ffineq}). Indeed, since $d_1\in F_{ij}(k)$, we have $d_1\in D(k)$. By definition of $D(k)$ there were no edges $(i,j)$ during times
$\ell B, \ldots, k-1$ which crossed the cut $d_1$. Applying Corollary \ref{order}, we have that $f_i'(x_i(k)) )  \geq f_{d_1}'(x_{d_1}(\ell B))$ and that $f_j'(x_j(k))  \leq f_{d_q+1}'(x_{d_q+1})(\ell B)$.
Therefore, \begin{small} \vspace{-0.22cm}\begin{align}
&f_i'(x_i({k}))-f_j'(x_j({k}))\nonumber\\
&\geq f_{d_1}'(x_{d_1}(\ell B))-f_{{d_q}+1}'(x_{{d_q}+1}(\ell B))\nonumber\\
& \geq  \sum_{d\in F_{ij}({k})} f_d'(x_d(\ell B))-f_{d+1}'(x_{d+1}(\ell B)).\nonumber
\end{align} \vspace{-0.3cm} \end{small} 
This implies that \vspace{-0.2cm} \begin{small}
\begin{align*} & \left[ f_i'(x_i({k}))-f_j'(x_j({k})) \right]^2 \cr & \geq \sum_{d\in F_{ij}({k})} \left[ f_d'(x_d(\ell B))-f_{d+1}'(x_{d+1}(\ell B)) \right]^2
\end{align*} \end{small}\vspace{-0.2cm}
A consequence of this last inequality is that \begin{small} \begin{align}
&\sum_{\!\!\!\!k=\ell B}^{\!\!(\ell +1)B-1}\!\!\!\!\!\!\!\!~~\sum_{(i,j)\in
\overline{\mathcal{E}}({k})}\!\!\!\!\!\!~(f_i'(x_i(k))-f_j'(x_j(k)))^2\nonumber\\
&\geq\!\!\!\!\!\!\!\!\!\sum_{\!\!\!\!{k=\ell B}}^{(\ell +1)B-1}\!\!\!\!\!\!\!\!~~\sum_{(i,j)\in
\overline{\mathcal{E}}(k)}\!\!~~\sum_{d\in
F_{ij}(k)}\!\!\!\!\!~(f_d'(x_d(\ell B))-f_{d+1}'(x_{d+1}(\ell B)))^2\nonumber\\
&=\!\!\!\!\!\!\sum_{{k=\ell B}}^{(\ell +1)B-1}\!\!\!\!\!\!~~\sum_{d\in D(k)} ~(f_d'(x_d(\ell B))-f_{d+1}'(x_{{d}+1}(\ell B)))^2\nonumber\\
&=\sum_{d=1}^{n-1}(f_d'(x_d(\ell B))-f_{d+1}'(x_{d+1}(\ell B)))^2\nonumber,
\end{align} \end{small} where the final equality used the fact that every $d \in \{1, \ldots, n-1\}$ such that $f_{d}'(x_d(\ell B)) - f_{d+1}'(x_{d+1}(\ell B)) \neq 0$ appears in some
$D(k)$, which is a restatement of  Lemma \ref{prop:BConnectivityRelaxation}.
\end{proof}

Our next lemma is a refinement of Lemma \ref{lem:BoundDiffGradOnBInterval} which yields more convenient bounds.

\begin{lemma}\label{lem:LowerBoundOnConvergenceMeasure} \begin{small}
\begin{align}
&\sum_{k=\ell B}^{(\ell +1)B-1}\sum_{(i,j)\in
\overline{\mathcal{E}}({k})}(f_i'(x_i(k))-f_j'(x_j(k)))^2\nonumber\\
&\geq\frac{1}{n^2}\sum_{i=1}^{n}(f_i'(x_i(\ell B))-f_i'(x_i^*))^2.
\end{align} \end{small}
\end{lemma}

\begin{proof}
By Lemma \ref{lem:BoundDiffGradOnBInterval}, if we relabel the nodes so that
$f_1'(x_1(\ell B)), f_2'(x_2(\ell B)), \ldots, f_n'(x_n(\ell B))$ is nonincreasing, \begin{small}
\begin{align}
\frac{\sum_{k=\ell B}^{(\ell +1)B-1}\sum_{(i,j)\in
\overline{\mathcal{E}}(k)}(f_i'(x_i(k))-f_j'(x_j(k)))^2}{\sum_{i=1}^n(f_i'(x_i(\ell B))-f_i'(x_i^*))^2}\nonumber\\
\geq \frac{\sum_{i=1}^{n-1}(f_i'(x_i(\ell B))-f_{i+1}'(x_{i+1}(\ell B)))^2}{\sum_{i=1}^n(f_i'(x_i(\ell B))-f_i'(x_i^*))^2}.\nonumber
\end{align} \end{small}
Let $q = f_1'(x_1^*)$; by Proposition 1, we have that $q=f_i'(x_i^*)$ for all $i=1, \ldots, n$. Define $g_i(z) = f_i(z) - q z$. We can then
rewrite the above inequality as \begin{small}
\begin{align}
\frac{\sum_{k=\ell B}^{(\ell +1)B-1}\sum_{(i,j)\in
\overline{\mathcal{E}}(k)}(f_i'(x_i(k))-f_j'(x_j(k)))^2}{\sum_{i=1}^n(f_i'(x_i(\ell B))-f_i'(x_i^*))^2}\nonumber\\ \geq
\frac{\sum_{i=1}^{n-1}(g_i'(x_i(\ell B))-g_{i+1}'(x_{i+1}(\ell B)))^2}{\sum_{i=1}^n g_i'^2(x_i(\ell B))}.\nonumber
\end{align} \end{small}
Clearly, the sequence $g_1'(x_1(\ell B)), \ldots, g_n'(x_n(\ell B)$ is in nonincreasing order. It therefore follows that \begin{footnotesize}
\begin{align}\label{eq:LowerBoundFractionMeasure}
&\frac{\sum_{k=\ell B}^{(\ell +1)B-1}\sum_{(i,j)\in \overline{\mathcal{E}}(t)}(f_i'(x_i(k))-f_j'(x_j(k)))^2}{\sum_{i=1}^n(f_i'(x_i(\ell B))-f_i'(x_i^*))^2}\nonumber\\
&\geq \min_{s_1\geq s_2\geq\ldots\geq s_n} \frac{\sum_{i=1}^{n-1}(s_i-s_{i+1})^2}{\sum_{i=1}^n s_i^2}.
\end{align} \end{footnotesize} Lemma 5 of \citep{NedicOlshevskyOzdaglarTsitsiklis09} shows that the right hand side is at least $1/n^2$. This immediately implies the lemma.
\end{proof}

We now turn to the statement and proof of our  main result. We will use $R_0$ to denote a measure of initial distance to an optimal solution defined as,
\begin{small} \begin{equation}
R_0\!=\!\sup_{\mathbf{x}\in\mathcal{S}:F(\mathbf{x})\leq
F(\mathbf{x}(0))}\!~~\sup_{\mathbf{x}^*\in\mathcal{X}^*}\!\!\|\mathbf{x}\;-~\mathbf{x}^*\|.\nonumber
\end{equation} \end{small} In words, $R_0$ is the largest distance to the set of optimal solutions from any point whose objective not larger than the objective at ${\bf x}(0)$.
Note that $R_0$ may not be finite, in which case part of our result below will be vacuously true.

Our main result is then the following theorem.

\begin{theorem}\label{theorem:QuadraticBound}  \begin{small}
\begin{equation} \label{thmlip} F(\mathbf{x}(k)) -  F(\mathbf{x}^*)  \leq frac{8LR_0^2n^2}{ \lfloor k/B \rfloor}, \end{equation} \end{small} 
where  $L=\max_{i=1, \ldots, n} L_i$ and $\lfloor z \rfloor$ denotes the largest integer which is at most $z$.

Furthermore, if all $f_i(\cdot)$ are $\mu$-strongly convex\footnote{A function $w: \mathbb{R}^m \rightarrow \mathbb{R}$ is called $\mu$-strongly convex if
$w(u) \geq w(v) + w'(v)^T (u-v) + (\mu/2) ||u-v||_2^2$ for all $u,v \in \mathbb{R}^m$. } for some $\mu \geq 0$, then we also have \begin{footnotesize} \begin{equation} \label{thmmu}
F(\mathbf{x}(k)) -  F(\mathbf{x}^*)  \leq  \left( 1 - \frac{\mu}{4Ln^2} \right)^{\lfloor k/B \rfloor} \left( F({\bf x}(0)) - F({\bf x}^*) \right).
\end{equation} \end{footnotesize}
\end{theorem}

\begin{proof}
By Lemma \ref{lem:DecreasingFuncValue} we have, \smallskip \begin{footnotesize}
\begin{align}\label{eq:BfDecreasingLowerBound}
&F(\mathbf{x}((\ell +1)B))\nonumber\\&\leq \!F(\mathbf{x}(\ell B))-\!\!\sum_{k=\ell B}^{(\ell +1)B-1}\!\sum_{(i,j)\in
\overline{\mathcal{E}}({k})}\!\!\!\!\!\!\frac{(f_j'(x_j({k}))-f_i'(x_i({k})))^2}{4(L_i+L_j)}\nonumber\\
&\leq \!F(\mathbf{x}(\ell B))- \sum_{i=1}^{n}\frac{(f_i'(x_i(\ell B))-f_i'(x_i^*))^2}{8Ln^2},
\end{align} \end{footnotesize}
where the last step is due to Lemma \ref{lem:LowerBoundOnConvergenceMeasure} and the inequality $L_i+L_j\leq 2L$.

Next, since $F$ is convex we have \begin{footnotesize}
\begin{align}
F({\bf x}^*)-F(\mathbf{x}&(\ell B))
\geq \langle\nabla F(\mathbf{x}(\ell B)),\mathbf{x}^*-\mathbf{x}(\ell B)\rangle\nonumber\\
&=
\langle \nabla F(\mathbf{x}(\ell B))-\nabla F(\mathbf{x}^*),\mathbf{x}^*-\mathbf{x}(\ell B)\rangle\nonumber,
\end{align} \end{footnotesize}
where the last equality follows since, by Proposition \ref{prop:OptimalCondition}, the components of $\nabla F({\bf x}^*)$ are identical and since $\mathbf{x}(\ell B),\;\mathbf{x}^*\in S$, we have that the entries of
${\bf x}^* - {\bf x}(\ell B)$ sum to zero. Next, negating both sides of the above equation and using Cauchy-Schwarz,
\begin{small} \begin{align}
F(\mathbf{x}(\ell B))\!-\!F({\bf x}^*)\!\!
&\leq R_0\|\nabla F(\mathbf{x}(\ell B))-\nabla F(\mathbf{x}^*)\|_2,\label{eq:fDistancetof*}
\end{align} \end{small} where we used  that $F({\bf x}(k))$ is nonincreasing.

Combining Eqs. \eqref{eq:BfDecreasingLowerBound} and \eqref{eq:fDistancetof*} we have \begin{footnotesize}
\begin{align} \label{Fbdec}
&F(\mathbf{x}((\ell +1)B))\!-\!F({\bf x}^*)\nonumber\\
&\leq
F(\mathbf{x}(\ell B))\!-\!F({\bf x}^*)\!-\!\sum_{i=1}^{n}\frac{(f_i'(x_i(\ell B))-f_i'(x_i^*))^2}{8Ln^2}\nonumber\\
&\leq F(\mathbf{x}(\ell B))\!-\!F({\bf x}^*)\!-\!\frac{(F(\mathbf{x}(\ell B))-F({\bf x}^*))^2}{8Ln^2R_0^2}.
\end{align} \end{footnotesize} We now show the last inequality implies Eq. (\ref{thmlip}) via some standard equation manipulations. Letting  $\Delta(k) = F(\mathbf{x}(k))-F({\bf x}^*)$, note that $\Delta(k)$ is nonincreasing by Lemma \ref{lem:DecreasingFuncValue}. We have just shown
\begin{align}
\Delta((\ell +1)B) \leq \Delta(\ell B) - \frac{\Delta^2(\ell B)}{8LR_0^2n^2}.\nonumber
\end{align}
Dividing both sides of this by $\Delta((\ell +1)B)\Delta(\ell B)$ and rearranging, we obtain \begin{footnotesize}
\begin{align}
\frac{1}{\Delta((\ell +1)B)}&\geq\frac{1}{\Delta(\ell B)}+\frac{1}{8LR_0^2n^2}\frac{\Delta(\ell B)}{\Delta{((\ell +1)B)}}\nonumber\\
&\geq\frac{1}{\Delta(\ell B)}+\frac{1}{8LR_0^2n^2},\nonumber
\end{align} \end{footnotesize} where we used the monotonicity of $\Delta(k)$.
Summing up this inequality up over $\ell =0, \ldots, t-1$, we obtain $F(\mathbf{x}(tB))-F({\bf x}^*) \leq\frac{8LR_0^2n^2}{t}$ and using the monotonicity of $F({\bf x}(k))$ we obtain Eq. (\ref{thmlip}).

Turning now to Eq. (\ref{thmmu}), let us define as before $g_i(x) = f_i(x) - q x$ where $q = f_1'(x_1^*)$, and further let $G({\bf x}) = \sum_{i=1}^n g_i(x_i)$. Observe that $G({\bf x})$ is $\mu$-strongly convex and has global minimizer at ${\bf x}^*$. Consequently if ${\bf x} \in \mathcal{S}$, \begin{small}
\begin{align*} \sum_{i=1}^n (f_i'(x_i) - f_i'(x_i^*))^2  = || \nabla G({\bf x})||_2^2 & \geq 2 \mu (G({\bf x}) - G({\bf x}^*)) \cr
& = 2 \mu (F({\bf x}) - F({\bf x}^*) ),
\end{align*} \end{small} where the final equality used the fact that the sum of the entries of ${\bf x}$ and ${\bf x}^*$ is the same since both are in $\mathcal{S}$.
Thus from Eq. (\ref{Fbdec}), \begin{footnotesize}
\begin{align*}
&F(\mathbf{x}((\ell +1)B))\!-\!F({\bf x}^*) \\
&\leq
F(\mathbf{x}(\ell B))\!-\!F({\bf x}^*)\!-\!\sum_{i=1}^{n}\frac{2 \mu \left( F({\bf x}(\ell B)) - F({\bf x}^*) \right) }{8Ln^2}
\end{align*} \end{footnotesize} which immediately implies Eq. (\ref{thmmu}).
\end{proof} 

\begin{remark}   Note that although Eq. (\ref{prob:MainProblem}) does not have constraints on the variables $x_i$, for certain functions $f_i(x_i)$ our algorithm automatically solves a 
constrained version of the problem. For example, if the initial conditions $x_i(0)$ are all nonnegative and $f_i'(0)=f_j'(0)$ for all $i,j$, then by Lemma \ref{mono} the 
constraint $x_i \geq 0$ will automatically be satisfied throughout the execution of the gradient balancing method.  In other words, the  constraints $x_i \geq 0$ can be added ``for free.'' 
The condition on the functions $f_i(x)$ is somewhat restrictive, but admissible $f_i$ include, for example, all polynomials with nonnegative coefficients whose linear coefficient is zero.
\end{remark}

\vspace{-0.4cm}
\section{A simulation}\label{sec:Simulation} 
We now describe a simulation of the gradient balancing protocol on some particular graphs. We consider local objective functions $f_i(x_i)=w_i(x_i-a_i)^4$ where the nonnegative coefficient $w_i$ and  the coefficient $a_i$ are chosen uniformly on $[0,1]$. We set $K=0$. We show simulations of the line and lollipop graphs in Figure \ref{fig:LineSimulation} and Figure \ref{fig:LollipopSimulation}, respectively, where we plot the first time $F({\bf x}(k)) - F({\bf x}^*) < 1/100$ on the y-axis. The figures appear to be broadly consistent with the quadratic bound of Theorem \ref{theorem:QuadraticBound}.

\begin{figure}[t]
\leftskip -0.2cm
\includegraphics[width=3.0in,height=1.6in]{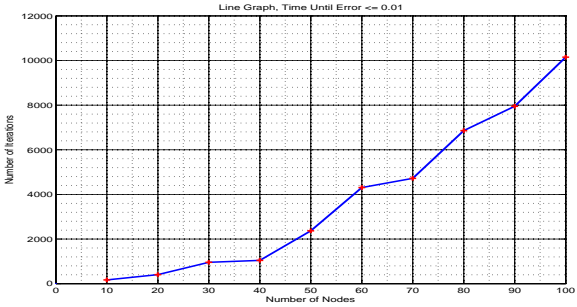}
\caption{Convergence time for  gradient balancing  as a function of the number of nodes for the line graph.}
\label{fig:LineSimulation}
\end{figure}
\begin{figure}[t]
\leftskip -0.2cm
\includegraphics[width=3.0in,height=1.6in]{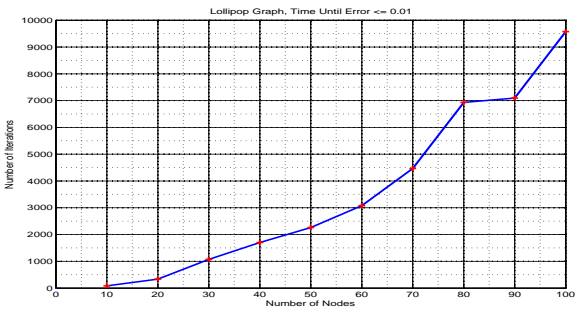}
\caption{Convergence time for  gradient balancing  as a function of the number of nodes for the lollipop graph.}
\label{fig:LollipopSimulation}
\end{figure}
\section{Concluding remarks}\label{sec:Conclusion}
We have proposed a new deterministic, distributed method for the network resource allocation problem with convergence time that scales quadratically in $n$ on time-varying undirected graphs. The main open questions left by this work include improving the convergence time and obtaining similar convergence rate guarantees in the distributed setting with local constraints at each node.

\bibliography{refs}
\bibliographystyle{plain}
\end{document}